\documentclass[runningheads]{llncs}
\usepackage[T1]{fontenc}
\usepackage{graphicx}
\usepackage{amsmath,amssymb, color}
\newtheorem{assumption}[theorem]{Assumption}
\begin{document}
\title{Fixed Point Theory Analysis of a Lambda Policy Iteration with Randomization for the \'Ciri\'c Contraction Operator}
\titlerunning{Fixed Point Theory Analysis of a Lambda Policy Iteration}
\author{Abdelkader Belhenniche\inst{1}\orcidID{0000-0003-4272-6232} \and \\
Roman Chertovskih\inst{1}
\orcidID{0000-0002-5179-4344}
}
\authorrunning{A. Belhenniche and R. Chertovskih}
\institute{Faculty of Engineering, Research Center for Systems and Technologies (SYSTEC), ARISE, University of Porto  Rua  Dr. Roberto Frias s/n, 4200-465 Porto, Portugal \email{belhenniche@fe.up.pt, roman@fe.up.pt}}

\maketitle

\begin{abstract}
We apply methods of the fixed point theory to a Lambda policy iteration with a randomization algorithm for weak contractions mappings. This type of mappings covers a broader range than the strong contractions typically considered in the literature, such as \'Ciri\'c contraction. Specifically, we explore the characteristics of reinforcement learning procedures developed for feedback control within the context of fixed point theory. Under relatively general assumptions, we identify the sufficient conditions for convergence with a probability of one in infinite-dimensional policy spaces.
\end{abstract}
\section{Introduction}
\quad In \cite{B54} Richard Bellman introduced the concept of dynamic programming. His main idea was to arrange a problem into easier embedded sub-problems which can be solved recurrently in time. The influence of dynamic programming techniques in optimization-based feedback control was extremely important, and, since then, it has been extended to many  control problems, particularly, impulsive control (see \cite{Fraga2011}, \cite{AJP03}, and references therein).

\medskip

\quad  An important class of these iterative schemes is designated by Reinforcement Learning (RL) algorithms. Bertsekas and Tsitsikilis in \cite{Bertsekas1996} offer a broad range of RL algorithms in the framework of the Value Iteration (VI) and Policy Iteration (PI) methods. 

 \medskip
 
 \quad In \cite{Ber-Ioffe1996}, Bertsekas and Ioffe provided the analysis of Temporal Differences policy iteration ($ TD( \lambda) $) scheme in the context of Neuro-Dynamic Programming framework in which they show that $ TD( \lambda) $ scheme can be embedded into a $ PI $ scheme, denoted by $ \lambda $-PIR. In \cite{Bertsekas2018b}, Bertsekas explores the connection between $ TD( \lambda) $ and proximal algorithms, which are more suitable for solving convex and optimization problems.

 \medskip
 
\quad In \cite{Li2020}, Yachao et al. use abstract DP models and extend the $ \lambda $-PIR scheme for finite policy problems to contractive models with infinite policies. The authors construct a well-posed and compact operator, playing a critical role in the algorithm, and defining the convergence conditions with probability one of the $ \lambda $-PIR scheme for problems with infinite-dimensional policy spaces. In \cite{ASF}, Belhenniche et al. investigate the properties of reinforcement learning procedures  developed for feedback control in the framework of fixed-point theory. Also, under general assumptions, the authors determine sufficient convergence conditions with probability one in infinite-dimensional policy spaces.

\medskip

Fixed-point theory is an effective instrument in topology, nonlinear analysis, optimal control, and machine learning. The well-known Banach contraction states that where $ (X, d) $ is a complete metric space, and $ T: X \rightarrow X $ is a mapping satisfying
\begin{equation}\label{BanachContraction}
 d(T(x), T(y)) \leq \gamma d(x, y),
\end{equation}
 for some $ \gamma \in (0, 1) $ and for all $ x, y \in X $, then $ T $ has a unique fixed point $ x^{*}$ and the sequence $\{x_n\}$ generated by the iterative process $ x_{n+1}= Tx_{n} $  converges to $ x^{*} $ for some $ x \in X $. The generalization of the Banach contraction has been widely studied in multiple settings. Ćirić   \cite{Cir74} presented a totally independent type of contraction, leading to a unique fixed point.
 
 \medskip
 

\quad In this article, we apply methods of fixed point theory to investigate a Lambda policy iteration with a randomization algorithm for mappings that are merely weak contractions. We use the Ciric's operator to guarantee the existence and uniqueness given by a fixed-point theorem, as well as the convergence of the considered operators. The practical relevance of this result is that it can be applied even to discontinuous operators. Recall that, a mapping $F: X\to X$ is called to be a \'Ciri\'c, if there exists $ \gamma \in (0, 1) $ , such that: 
 \begin{equation}
d(Fx,Fy)\leq\gamma\max\left(  
d(x,y),d(x, Fx), d(y, Fy),\frac{1}{2}(d(x,Fy)
		+d(y,Fx))
  \right)
\label{b(1)}
\end{equation}%
	for all $x$, $y\in X$,
\newline

The paper is organized as follows. In the next section, we formulate the iterative feedback control problem to be studied further. In the next section we present several results related to the fixed-point theory, which are fundamental for the further development of the study. In section \ref{Main}, we summarize some of the implications of the \'Ciri\'c  contraction assumptions.  The $\lambda$-policy iteration with randomization scheme is formulated in section \ref{sec:Randomization} for the \'Ciri\'c operators. In section \ref{sec:Convergence+Stability}, we prove the convergence in norm with probability one of the iterative procedure. In the last section the conclusions and our plans for future work are presented.

\section{Preliminaries and Auxiliary Results}\label{Perlim}
We consider a set $X$ of states, a set $U$ of controls, and, for each $x \in X$, a nonempty control constraint $U(x) \subset U$. We denote by $M$ the set of all functions $\mu : X \rightarrow U$ with $\mu(x) \in U(x)$ for all $x \in X$, which in the  analysis below will be referred to as a policy.

We denote by $\upsilon(X)$ the set of functions $V: X \rightarrow \mathbb{R}$ and by $\bar{\upsilon}(X)$ the set of functions $V: X \rightarrow \overline{\mathbb{R}}$, where $\overline{\mathbb{R}} = \mathbb{R} \cup \{-\infty, \infty\}$.

We study the operator of the form
$$H: X \times U \times \upsilon(X) \rightarrow \mathbb{R}. $$
and, for each policy $\mu \in M$, we consider the mapping $F_{\mu}: \upsilon(X) \rightarrow \upsilon(X)$ defined by
$$F_{\mu}V(x) := H(x, \mu(x), V), \quad \forall x \in X.$$
and a mapping $F: \upsilon(X) \rightarrow \bar{\upsilon}(X)$ defined by
$$ FV(x) := \inf_{\mu \in M} \{F_{\mu}V(x)\}, \quad \forall x \in X.$$
Using the definitions presented above, we find:
$$ FV(x) = \inf_{\mu \in M}\{H(x, \mu(x), V)\} = \inf_{u \in U}\{H(x, u, V)\}.$$

For a given positive function $\nu : X \rightarrow \mathbb{R}$, we denote by $B(X)$ the set of functions $V$ such that $\| V \| < \infty$, where the norm $\| \cdot \|$ on $B(X)$ is defined by
$$
\| V \| = \sup_{x \in X} \left\{\frac{| V(x) |}{\nu (x) }\right\}.
$$

It is evident that the following lemma holds:

\begin{lemma}\label{Lemma 1}
$B(X)$ is complete with respect to the topology induced by $\| \cdot \|$.
\end{lemma}

It is also clear that $B(X)$ is closed and convex. Thus, given $\{V_{k}\}_{k=1}^\infty \subset B(X)$ and $V \in B(X)$, if $V_{k} \rightarrow V$ in the sense that $\displaystyle\lim_{k \rightarrow \infty} \| V_{k} - V \| = 0$, then $\displaystyle\lim_{k \rightarrow \infty} V_{k}(x) = V(x)$ for all $x \in X$.

Next, we introduce the following standard assumptions.

\begin{assumption}\label{Assumption 1}
(Well-posedness) For all $V \in B(X)$ and $\forall \mu \in M$, $F_{\mu} V \in B(X)$ and $FV \in B(X)$.
\end{assumption}

\begin{definition}\label{Definition 1} (\'Ciri\'c operator)
A self-map $F_{\mu}$ of $B(X)$ is called a \'Ciri\'c  operator if for all $V, V^{'} \in B(X)$ there exists $k \in (0, 1)$ such that:
\begin{align}\label{KC1}
\begin{split}
\| F_{\mu}V - F_{\mu}V^{'}\| \leq & k \max \Bigl( \| V - V^{'} \|, \| V -  F_\mu V \|, \| V^{'} -  F_\mu V^{'}  \|,\\
&\frac{1}{2} (\| V -  F_\mu V^{'}  \|+ \| V^{'}   -  F_\mu V \|)\Bigr)
\end{split}
\end{align}
\end{definition}


\begin{assumption}\label{Assumption 2}
The self-map $F$ is a \'Ciri\'c operator.
\end{assumption}

\begin{example}\label{Example 1}
To see the relevance of this extension considered, consider the following example, where the \'Ciri\'c mapping is not a
contraction. Let $ T:X \to X$, defined by
\begin{align*}
Tx = \begin{cases}
\frac{x}{4} & \text{for } x \in [0, 1]\\
\frac{x}{5} & \text{for } x \in (1, 2]
\end{cases}
\end{align*}

The mapping $T $ is of a \'Ciri\'c type with $ \sigma = \displaystyle{\frac{1}{4}} $. Indeed, if both $ x $ and $ y $ are in $ X_{1}= [0, 1]$ or in $ X_{2} =(1, 2] $, then:
$$ d(Tx, Ty) \leq \frac{1}{4}\left| x-y \right|   $$
Now let $ x $  in $ X_{1}= [0, 1]$  and $y$ in $ X_{2} =(1, 2] $ then,
\\
for $ x >\frac{4}{5}y $
$$ d(Tx, Ty) =\frac{1}{4}\left| x-\frac{4}{5}y \right| \leq  \frac{1}{4}\left| x-\frac{1}{5}y \right| = \frac{1}{4} d(x, Ty);$$
for  $ x < \frac{4}{5}y $
$$ d(Tx, Ty) =\frac{1}{4}\left| \frac{4}{5}y-x\right| \leq  \frac{1}{4}\left| y-x \right| \leq \frac{1}{4} d(y, x).
$$
Therefore, $T$ satisfies the condition:
$$  d(Tx, Ty) \leq \frac{1}{4} \max \left( d(x, y) , d(x, Ty), d(y, Tx) \right) $$
and, hence, the inequality \eqref{b(1)} holds true.\\

Here, $F$ is discontinuous at $x=1$, consequently, Banach contraction is not satisfied. A \'Ciri\'c  operator doesn't have to be continuous in general, however, it is always convergent to a fixed point.
\end{example}

In what follows some results to be used further are presented.

\begin{theorem}\label{Theorem 1}(Existence)
Let the operators $F$ and $F_{\mu} : B(X) \rightarrow B(X)$ be a Ćirić, then $F$ and $F_{\mu}$ have, respectively, $V^{*}$ and $V_{\mu}$ as fixed points.
\end{theorem}

From Theorem \ref{Theorem 1}, it follows:

\begin{lemma}\label{Lemma 2}
The following holds:
\begin{itemize}
\item[i)] For an arbitrary $V_0 \in B(X)$, the sequence $\{V_k\}$, where $V_{k+1} = F_{\mu} V_k$, converges in norm to $V_{\mu}$.
\item[ii)] For an arbitrary $V_0 \in B(X)$, the sequence $\{V_k\}$, where $V_{k+1} = FV_k$, converges in norm to $V^{*}$.
\end{itemize}
\end{lemma}

We will assume the following two properties to hold.

\begin{assumption}\label{Assumption 3} (Monotonicity) For all $V, V^{'} \in B(X)$, we have that $V \leq V^{'}$ implies
\[ H(x, u, V) \leq H(x, u, V^{'} ), \quad \forall x \in  X, \; u \in U(x) \]
where $\leq$ is in a pointwise sense in $X$.
\end{assumption}

\begin{assumption}\label{Assumption 4} (Attainability)
For all $V \in B(X)$, there exists $\mu \in M,$ such that $F_{\mu} V = FV$.
\end{assumption}

\section{Main Results}\label{Main}

In this section we introduce a proposition summarizing some fundamental consequences of the  \'Ciri\'c contraction assumptions.

\begin{proposition}\label{PropMain}
Let $F,  F_{\mu}$ be a \'Ciri\'c contractions satisfying Assumption \ref{Assumption 1}. Then the following statements hold:
\begin{enumerate}
    \item For any $V \in B(X)\colon$ 
    \begin{equation}\label{EQmain1}
        \Vert V^{*} - V \Vert \leq \gamma \Vert FV - V\Vert.
    \end{equation}
    
     \item For any $V \in B(X)$ and $\mu \in \mathcal{M}$:
    \begin{equation}\label{EQmain2}
         \Vert V_{\mu} - V \Vert \leq \gamma \Vert F_{\mu}V - V \Vert,
    \end{equation}
{\rm where} $\gamma = \max\left(\frac{2-\sigma }{2 - 2\sigma},  \frac{1 }{1 - \sigma}\right)$.
    
\end{enumerate}
\end{proposition}

\begin{proof}
Let us consider $V_{k} = FV_{k-1}$. We have:
\begin{align*}
\begin{split}
    \Vert FV_{k} - FV_{k-1} \Vert \leq & \sigma \max \bigl(\Vert FV_{k-1} - FV_{k-2} \Vert; \Vert FV_{k} - FV_{k-1} \Vert; \Vert FV_{k-1} - FV_{k-2} \Vert;\\
                                    &\frac{1}{2}(\Vert FV_{k-1} - FV_{k-1} \Vert+\Vert FV_{k-2} - FV_{k} \Vert) \bigr)\leq  \sigma \max S_{k}
\end{split}
\end{align*}

where $$   S_{k}= \bigl(\Vert FV_{k-1} - FV_{k-2} \Vert; \Vert FV_{k} - FV_{k-1} \Vert; \frac{1}{2}(\Vert FV_{k-2} - FV_{k} \Vert) \bigr).$$
There are three cases:\\
{\bf First case}: $\max  S_{k}=\Vert FV_{k-1} - FV_{k-2} \Vert $. Then we get:
$$
\Vert FV_{k} - FV_{k-1} \Vert  \leq \sigma \Vert FV_{k-1} - FV_{k-2} \Vert \leq \sigma^{k} \Vert FV - V \Vert.$$
Using the triangle inequality one finds:

\begin{align*}
\begin{split}
    \Vert FV_{k} - V \Vert  &\leq \sum_{n=1}^{k} \| FV_{n} - FV_{n-1} \| \\
                            &\leq  \sum_{n=1}^{k} \sigma^{n-1} \| FV - V \|.
\end{split}
 \end{align*}

Taking the limit as $k \rightarrow +\infty$ and using Lemma \ref{Lemma 2}, we get:
$$
\Vert V^{*} - V\Vert \leq \frac{1 }{1 - \sigma} \Vert FV - V \Vert.
$$
{\bf Second case}: $\max S_{k} =  \Vert FV_{k} - FV_{k-1} \Vert $. This leads us to contradiction as $\sigma < 1$.
\\ {\bf Third case}: $\max S_{k} = \frac{1}{2} \Vert FV_{k-2} - FV_{k} \Vert$. This yields:
$$
\Vert FV_{k} - FV_{k-1} \Vert  \leq  \frac{1}{2}\sigma \Vert FV_{k-2} - FV_{k} \Vert.
$$
$$   \leq  \frac{1}{2}\sigma (\Vert FV_{k-2} - FV_{k-1} \Vert +\Vert FV_{k-1} - FV_{k-2} \Vert).$$
$$   \leq  \frac{\sigma}{2-\sigma}(\Vert FV_{k-2} - FV_{k-1} \Vert $$
Using the triangle inequality one finds:

\begin{align*}
\begin{split}
	\Vert FV_{k} - V \Vert  &\leq \sum_{n=1}^{k} \| FV_{n} - FV_{n-1} \| \\
	&\leq  \sum_{n=1}^{k} \dfrac{\sigma}{2-\sigma}^{n-1} \| FV - V \|.
 \end{split}
\end{align*}
  
Taking the limit as $k \rightarrow +\infty$ and using Lemma \ref{Lemma 2}, we get:
$$
\Vert V^{*} - V\Vert \leq \frac{2-\sigma }{2 - 2\sigma} \Vert FV - V \Vert.
$$
Thus,
$$\Vert V^{*} - V\Vert \leq \gamma \Vert FV - V \Vert. $$
such that: $\gamma = \max \left(\frac{2-\sigma }{2 - 2\sigma},  \frac{1 }{1 - \sigma}\right).$
The proof for the statement 2. can be constructed analogously.
\end{proof}

\begin{remark}\label{Remark1}
If we take $V = V^{*}$, Proposition \ref{PropMain} shows that for $\epsilon > 0$, there exists $\mu_{\epsilon} \in \mathcal{M}$, such that $\Vert V_{\mu_{\epsilon}} - V^{*}\Vert \leq \epsilon$, which may be obtained by letting $\mu_{\epsilon}(x)$ minimize $H(x, u, V^{*})$ over $U(x)$ within an error of $\gamma \epsilon \nu(x)$ for all $x \in X$. 
Indeed, we have:

$$
     \Vert V_{\mu_{\epsilon}} - V^{*}\Vert \leq \gamma\Vert F_{\mu_{\epsilon}}V^{*} - V^{*}\Vert = \gamma \Vert F_{\mu_{\epsilon}}V^{*} - FV^{*}\Vert \leq \epsilon. 
 $$

Thus,
\[
\vert F_{\mu_{\epsilon}}V^{*}(x) - FV^{*}(x)\vert \leq \gamma \epsilon \nu(x).
\]
The importance of monotonicity and  contraction is demonstrated by showing that $V^{*}$, the fixed point of $F$, is the infimum over $\mu \in \mathcal{M}$ of $V_{\mu}$, the unique fixed point of $F_{\mu}$.
\end{remark}

\begin{proposition}
Let assumptions \ref{Assumption 1}, \ref{Assumption 2} and  \ref{Assumption 3} to hold true. Then:
\[
V^{*} (x) = \inf_{\mu \in \mathcal{M}} V_{\mu}(x) \quad \forall x \in X.
\]
\end{proposition}

\begin{proof}
First, let us show that $V^{*}(x) \leq \displaystyle\inf_{\mu \in \mu \mathcal{M}} V_{\mu}(x)$. For all $\mu \in \mathcal{M}$, we have $FV^{*} \leq F_{\mu} V^{*}$, and since $FV^{*} = V^{*}$, it follows that $V^{*}  \leq F_{\mu} V^{*}$. By applying repeatedly $F_{\mu}$ to both sides of the inequality, and by using the monotonicity assumption, we obtain $V^{*} \leq F_{\mu}^{k} V^{*}$ for all $k > 0$. If $k$ tends to infinity, then $V^{*} \leq V_{\mu}$ for all $\mu \in \mathcal{M}$.

To show that $V^{*}(x) \geq \displaystyle\inf_{\mu \in \mu \mathcal{M}} V_{\mu}(x)$, we apply Remark \ref{Remark1} and get that for all $\epsilon > 0$: $V_{\mu_{\epsilon}} \leq V^{*} + \epsilon$ and hence $V^{*}(x) \geq \displaystyle\inf_{\mu \in \mu \mathcal{M}} V_{\mu}(x)$ for all $x \in X$.
\end{proof}

\section{$\lambda$-Policy Iteration With Randomization }\label{sec:Randomization}

The $ \lambda $-PIR algorithm is a policy iteration iterative procedure with randomization that has been studied by many  authors \cite{Bertsekas2018a,Li2020,ASF}, being infinite policies considered in the later. 



Given some $ \lambda \in [0, 1) $, consider the mappings $ F_\mu^\lambda\in B(X) $ defined pointwisely by:
$$ F_\mu^\lambda V(x)= (1-\lambda) \sum_{l=0}^\infty \lambda^l (F_\mu^{l})V(x).$$
In what follows, we refer to this operator as $\lambda$-operator. For a given $ V_k \in B(X) $, and $ p_k \in (0,1) $, the current policy $\mu$, and the next cost approximation $ V_{k+1}$ are computed as follows:
\begin{eqnarray}
&& F_\mu V_k = F V_k \label{Step 1}\\
&& V_{k+1}=\left\{\begin{array}{ll}
F_\mu V_k & \mbox{ w.p. } p_k \vspace{.1cm}\\
F_\mu^\lambda V_k & \mbox{ w.p. } 1- p_k
\end{array}\right.\label{random-iter-policy}
\end{eqnarray}
Here, ''w.p.'' stands for ``with probability''. The following result holds:
\begin{theorem}\label{Theorem A} Let $  F_\mu $ be a  \'Ciri\'c self operator on $ B(X) $. Then, the function $ F_\mu^\lambda $ satisfying the following properties:
\begin{itemize}
\item $F_\mu^\lambda$ is well defined.
\item $F_\mu^\lambda$ is well posed.
\item $F_\mu^\lambda$ is a \'Ciri\'c's contraction with $  \displaystyle \rho = \sum_{l=1}^\infty (1- \lambda) \lambda^{l}  k^{l}$.
\end{itemize}
\end{theorem}

\begin{proof}
\begin{itemize}

Let us consider the sequence $ \left\{ \phi_{l}\right\}_{k=1}^{\infty} $ defined by $\displaystyle \phi_{l} = \sum_{l=1}^{k} \alpha_{l } (F_{\mu}^{l}V ) (x) $, where $ \alpha_{l }= (1- \lambda) \lambda^l.$ From lemma \ref{Lemma 2}, we have that $ F_{\mu}^{l} V(x) \rightarrow V_{\mu} (x) \in \mathbb{R}$ ($ \forall \epsilon > 0,  \vert F_\mu^\lambda  V_\mu(x)- V_\mu (x) \vert < \epsilon $), and, thus, $ \{ F_{\mu}^{l} V(x) \}_{l=1}^{\infty} $ is bounded.
Note that $F_\mu^\lambda  V_\mu = V_\mu$, otherwise, if $ V \neq V_\mu $, we would have

\begin{align*}\label{wdi}
\begin{split}
| F_\mu^\lambda V(x)- V_\mu(x) | \quad & =  \quad \left| \sum_{l=0}^\infty \alpha_l F_\mu^lV (x) - V_\mu (x)\right| \\
& =  \quad \left | \sum_{l=1}^\infty \alpha_l F_\mu^lV (x) - \sum_{l=1}^\infty \alpha_l F_\mu^l V_\mu(x) \right| \\
& \leq \quad \sum_{l=1}^\infty\alpha_l |F_\mu^l V (x)- F_\mu^l V_\mu (x)| \leq \sum_{l=1}^\infty\alpha_l \epsilon \leq \epsilon
\end{split}
\end{align*}
Therefore,  for all $ V \in B(X), x \in X $, the sequence $  \{ \phi_{l} \} $ converges in $ \mathbb{R} $. Next, we discuss the well-posedness of $F_\mu^\lambda$. 

Due to the fact that $F_\mu$ is a \'Ciri\'c  operator for all $l \in \mathbb{N}$, we have:
\begin{align*}
    \begin{split}
    \| F_\mu^{l}V  - F_\mu^{l} V_\mu \| \leq &  k\max ( \| F_\mu^{l-1}V - F_\mu^{l-1}V_\mu\|, \| F_\mu^{l-1} V - F_\mu^{l} V \|, \| F_\mu^{l-1} V_\mu  - F_\mu^{l-1} V_\mu \| ) ,\\
                                             & \frac{1}{2}(\| F_\mu^{l-1}V - F_\mu^{l}V_\mu\|+\| F_\mu^{l-1}V_\mu - F_\mu^{l}V\|) = k \max M(V, V_\mu).
     \end{split}
\end{align*}

\textbf{First case:} $\max M(V, V_\mu)= \| F_\mu^{l-1}V - F_\mu^{l-1}V_\mu\|$.
Then:
\begin{equation*}
     \| F_\mu^{l}V  - F_\mu^{l} V_\mu \| \leq  k( \| F_\mu^{l-1}V - F_\mu^{l-1}V_\mu\|) \leq k^{l}  \| V - V_\mu\|
\end{equation*}
   
\textbf{Second case:} 
$\max M(V, V_\mu)= \| F_\mu^{l-1} V - F_\mu^{l} V \|$.
Then:
\begin{equation*}
     \| F_\mu^{l}V  - F_\mu^{l} V_\mu \| \leq  k( \| F_\mu^{l-1}V - F_\mu^{l}V\|)
    \leq k^{l} \| V -  F_\mu V \|
\end{equation*}

\textbf{Third case:} $\max M(V, V_\mu)= \| F_\mu^{l-1} V_\mu - F_\mu^{l} V_\mu\|$. 
Then:
\begin{equation*}
       \| F_\mu^{l}V  - F_\mu^{l} V_\mu \| \leq  k\| F_\mu^{l-1} V_\mu - F_\mu^{l} V_\mu\|  \leq k^{l}  \| V_\mu  -  F_\mu V_\mu  \|
\end{equation*}

\textbf{Fourth case:}
$\max M(V, V_\mu)= \frac{1}{2}(\| F_\mu^{l-1}V - F_\mu^{l}V_\mu\|+\| F_\mu^{l-1}V_\mu - F_\mu^{l}V\|)$. Then:
\begin{align*}
    \begin{split}
    \| F_\mu^{l}V  - F_\mu^{l} V_\mu \| \leq  \frac{k}{2}(\| F_\mu^{l-1}V - F_\mu^{l}V_\mu\|+\| F_\mu^{l-1}V_\mu - F_\mu^{l}V\|) , \\
    \leq \frac{k^l}{2^l}(\| V -  F_\mu V_\mu \|+ \| V_\mu  -  F_\mu V \|) 
    \end{split}
\end{align*}
Hence,
\begin{align*}
 \begin{split}
    \| F_\mu^{l}V  - F_\mu^{l} V_\mu \| \leq  & k^l \max ( \| V - V_\mu\|, \| V -  F_\mu V \|, \| V_\mu  -  F_\mu V_\mu  \| , \\
    & \frac{1}{2^l} (\| V -  F_\mu V_\mu \|+ \| V_\mu  -  F_\mu V \|))  
 \end{split} 
\end{align*}
and, thus,
\begin{align*}
\begin{split}
     | F_\mu^{l}V (x)- F_\mu^{l} V_\mu (x) | \leq & k^l \max ( \| V - V_\mu\|,  \| V -  F_\mu V \|, \| V_\mu  -  F_\mu V_\mu  \| , \\
    &\frac{1}{2^l}(\| V -  F_\mu V_\mu \|+ \| V_\mu  -  F_\mu V \|) \nu (x). 
\end{split}
\end{align*}
and due to $$ | F_\mu^\lambda V(x)- V_\mu(x) | \leq   \sum_{l=1}^\infty \alpha_l |F_\mu^l V (x)- F_\mu^l V_\mu (x)| $$
Therefore, we have four cases:\\
\textbf{First case:}\\
\begin{align*}
    | F_\mu^\lambda V(x)- V_\mu(x)| \leq & \sum_{l=1}^\infty \alpha_l  k^{l} ( \| V - V_\mu\| )\nu (x)\leq \overline{k_{1}}(\| V - V_\mu\|  )\nu (x)  .
\end{align*}
Thus,
\begin{align*}
    \| F_\mu^\lambda V \| \leq \overline{k_{1}} (\| V - V_\mu\|  )\nu (x)+ \Vert V_{\mu}\Vert
\end{align*}

\textbf{Second case:}\\
\begin{align*}
\begin{split}
    | F_\mu^\lambda V(x)- V_\mu(x)| \leq  &\sum_{l=1}^\infty \alpha_l  k^{l} ( \| V -  F_\mu V \|)\nu (x)  \leq \overline{k_{1}}(\| V -  F_\mu V \|  )\nu (x)  . 
\end{split}  
\end{align*}
Thus,
\begin{align*}
    \| F_\mu^\lambda V \| \leq \overline{k_{1}} (\| V -  F_\mu V \|  )\nu (x)+ \Vert V_{\mu}\Vert
\end{align*}

\textbf{Third case:}\\
\begin{align*}
  | F_\mu^\lambda V(x)- V_\mu(x)| &\leq \sum_{l=1}^\infty \alpha_l  k^{l} ( \| V_\mu  -  F_\mu V_\mu  \| )\nu (x) \\
  &\leq \overline{k_{1}}(\| V_\mu  -  F_\mu V_\mu\| )\nu (x)  .
\end{align*}
Thus,
\begin{align*}
    \| F_\mu^\lambda V \| \leq \overline{k_{1}} ( \| V_\mu  -  F_\mu V_\mu  \| )\nu (x)+ \Vert V_{\mu}\Vert
\end{align*}
\textbf{Fourth case:}\\
\begin{align*}
\begin{split}
      | F_\mu^\lambda V(x)- V_\mu(x)| \leq &\sum_{l=1}^\infty \alpha_l  k^{l}  \frac{1}{2^l} (\| V -  F_\mu V_\mu \|+ \| V_\mu  -  F_\mu V \|) \nu (x)\\
   & \leq \overline{k_{1}} \frac{1}{2^l} (\| V -  F_\mu V_\mu \|+ \| V_\mu  -  F_\mu V \|)\nu (x).
\end{split} 
\end{align*}
Thus,
\begin{align*}
    \| F_\mu^\lambda V \| \leq \overline{k_{1}} ( \frac{1}{2^l} (\| V -  F_\mu V_\mu \|+ \| V_\mu  -  F_\mu V \|) )\nu (x)+ \Vert V_{\mu}\Vert
\end{align*}
   
Where $\overline{k_{1}}  $ is given by:
\begin{equation*}
    \overline{k_{1}} = \sup_{x \in X} \sum_{l=1}^\infty   \alpha_l  k^{l} 
\end{equation*}

Since $ V_{\mu}\in B(X) $, we obtain that $F_\mu^\lambda V\in B(X)$.

 $ F_{\mu}^{\lambda} $ is \'Ciri\'c contraction operator with the modulus 
 $ \displaystyle\rho = \sum_{l=1}^\infty   (1- \lambda) \lambda^{l} k^{l} $.\\
 \begin{align*}
\| F_{\mu}^{\lambda}V - F_{\mu}^{\lambda}V^{'}\| &\leq \left\| \sum_{l=1}^\infty  (1- \lambda) \lambda^{l} ( F_{\mu}^l V- F_{\mu}^l V^{'} ) \right\|\\
& \leq \sum_{l=1}^\infty (1- \lambda) \lambda^{l} \| ( F_{\mu}^l V- F_{\mu}^l V^{'} ) \|\\
 & \leq \sum_{l=1}^\infty   (1- \lambda) \lambda^{l}  k^l \max ( \| V - V^{'}\|, \\
 & \quad \| V -  F_\mu V \|, \| V^{'}  -  F_\mu V^{'} \| , \\
 & \quad \frac{1}{2^l} (\| V -  F_\mu V^{'}\|+ \|  V^{'} -  F_\mu V \|))\\
 & \leq \rho \max ( \| V - V^{'}\|,\| V -  F_\mu V \|, \| V^{'}  -  F_\mu V^{'}  \| ) \\
 & \quad + \frac{1}{2} (\| V -  F_\mu V^{'} \|+ \|V^{'}   -  F_\mu V \|))
\end{align*}
\end{itemize}
\end{proof}


Alternatively, the monotonicity of $ F_{\mu}^{\lambda} $ can be derived from the monotonicity property inherent in $ F_{\mu} $. This is summarized by the following lemma.
\begin{lemma}\label{Lemma 3}
Let $ F_\mu : B(X) \rightarrow B(X) $  satisfy the definition \ref{Definition 1} and  Assumption \ref{Assumption 3}. Then the mappings  $ F_{\mu}^{\lambda} $ are monotonic in the sense that
$$  V \leq V^{'} \rightarrow F_{\mu}^{\lambda}V  \leq F_{\mu}^{\lambda}V^{'}\;\; \forall\; x \in X , \;\mu \in M. $$
\end{lemma}
\section{Convergence Of The  $\lambda$-PIR Algorithm}\label{sec:Convergence+Stability}

In this section, we establish the convergence of the $\lambda$-PIR algorithm under condition weaker than those considered so far. The main result of this section reads as follows:
\begin{theorem}\label{ConvTheor}
Let assumptions \ref{Assumption 1},\ref{Assumption 2},  \ref{Assumption 3} and definition \ref{Definition 1} hold, given $ V_{0} \in B(X)$ such that $ FV_{0} < V_{0} $, then the sequence $ \{ V_{k} \} $ generated by algorithm represented by \eqref{Step 1} and \eqref{random-iter-policy} converges in norm with probability one.
\end{theorem}

\begin{proof}
Since $ FV_{0} < V_{0} $, we have $ F_{\mu^{0}}V_{0}= FV_{0} < V_{0} $.

By monotonocity of $ F_{\mu^{0}} $, we have

$$ F_{\mu^{0}}^{l}V_{0} \leq F_{\mu^{0}}^{l-1}V_{0}, \quad F_{\mu^{0}}^{l}V_{0}\leq FV_{0},\quad \forall l  \in \mathbb{N}$$ which implies that
$$  F_{\mu^{0}}^{\lambda}V_{0} \leq F_{\mu^{0}}V_{0} \leq V_{0}. $$
This means that $ V_{1} $ is bounded from above with probability one by $ FV_{0} $. In addition, we also have $ V_{\mu^{0}} \geq V^{*} $ where $  V_{\mu^{0}} $ is the fixed point of both $ F_{\mu^{0}}^{\lambda} $, and $F_{\mu^{0}} $, thus, we have that
$$  V^{*} \leq V_{\mu^{0}}\leq F_{\mu^{0}}^{\lambda}V_{0} \leq F_{\mu^{0}}V_{0}. $$
This entails that $ V_{1} $ is bounded from below by $ V^{*} $ with probability one.

Now, due to that $ F $ is a \'Ciri\'c operator  we have:
$$ F^{2} V_{0}= F(F_{\mu^{0}}V_{0} )\leq F_{\mu^{0}}V_{0},$$
and, due to the monotonicity of $ F_{\mu^{0}}^{\lambda} $, and since $ F_{\mu^{0}}^{\lambda} $ and $ F_{\mu} $ can commute, we have that
$$ F(F_{\mu^{0}}^{\lambda} V_{0}) \leq  F_{\mu^{0}}(F_{\mu^{0}}^{\lambda} V_{0}) \leq  F_{\mu^{0} }^{\lambda}V_{0}. $$

Applying the iterations, we get that
$$ V^{*} \leq  V_{k} \leq F^{k} V_{0},$$
meaning that $ \{V_{k} \} $ converges in norm to $ V^{*} $ with probability one.
\end{proof}


\section{Conclusion}\label{Con}
In this paper, using techniques from the fixed point theory, we have extended the previously developed strongly contractive operators reinforcement learning process, denoted by Lambda policy iteration with randomization, to a broader class of contractive operators.  We have demonstrated the well-posedness, a \'Ciri\'c contractiveness, and well-definition of the Lambda policy iteration scheme in Banach spaces by using  a \'Ciri\'c fixed point properties. Our further steps in this line of research will be relate to the convergence characteristics of the corresponding computational algorithms, driving our efforts to produce outcomes that are comparable to those of generalised metric spaces.

\section*{Acknowledgments}
AB acknowledges the support of FCT for the grant 2021.07608.BD. RC acknowledges the financial support by the FCT  doi:10.54499/CEECINST/00010/2021/ CP1770/CT0006.
This work was also supported by multiple funding sources including the: Base funding (UIDB/00147/2020) and Programmatic funding (UIDP/00147/2020) of the SYSTEC -- Center for Systems and Technologies; ARISE - Associate Laboratory for Advanced Production and Intelligent Systems (LA/P/0112/2020); RELIABLE project (PTDC/EEI-AUT/3522/2020); MLDLCOV project (DSAIPA/CS/0086/ 2020), under the program IA 4 COVID-19: Data Science and Artificial Intelligence in Public Administration to strengthen the fight against COVID-19 and future pandemics in 2020; both funded by national funds through the FCT/MCTES (PIDDAC).

\end{document}